\documentclass[12pt, reqno]{amsart}
\usepackage{amsmath, amsthm, amscd, amsfonts, amssymb, graphicx, color}
\usepackage[bookmarksnumbered, colorlinks, plainpages]{hyperref}
\input{mathrsfs.sty}

\newtheorem{theorem}{Theorem}[section]
\newtheorem{lemma}[theorem]{Lemma}
\newtheorem{proposition}[theorem]{Proposition}

\theoremstyle{definition}

\theoremstyle{remark}

\numberwithin{equation}{section}

\begin{document}

\title[A characterization of inner product spaces]{A characterization of inner product spaces related to the $p$-angular distance}

\author[F. Dadipour, M.S. Moslehian]{F. Dadipour and M. S. Moslehian}

\address{Department of Pure Mathematics, Center of Excellence in
Analysis on Algebraic Structures (CEAAS), Ferdowsi University of
Mashhad, P. O. Box 1159, Mashhad 91775, Iran.}
\email{dadipoor@yahoo.com}
\email{moslehian@ferdowsi.um.ac.ir and
moslehian@ams.org}
\urladdr{\url{http://profsite.um.ac.ir/~moslehian/}}

\subjclass[2010]{Primary 46C15; Secondary 46B20, 46C05.}

\keywords{inner product space; characterization of inner product spaces; $p$-angular distance; Dunkl--Williams inequality.}

\begin{abstract}
In this paper we present a new characterization of inner product spaces related to the $p-$angular distance.
We also generalize some results due to Dunkl, Williams, Kirk, Smiley and Al-Rashed by using the notion of $p-$angular
distance.
\end{abstract} \maketitle

%------------------------------------------------------------------------------%

\section{Introduction}

In 1935, Fr\'echet \cite{FRE} gave a geometric characterization of inner product spaces. In the same year, Jordan and von Neumann
\cite{J-V} characterized inner product spaces as normed linear spaces satisfying the parallelogram law. In 1943, Ficken showed that
a normed linear space is an inner product space if and only if a reflection about a line in any two-dimensional subspace is an isometric mapping.
In 1947, Lorch presented several characteriztion of inner product spaces. Since then the problem of finding necessary and sufficient conditions for
a normed space to be an inner product space has been investigated by many mathematicians by considering some types of orthogonality or
some geometric aspects of underlying spaces. Some known characterizations of inner product spaces and their generalizations can be found in
\cite{A-T, Amir, B-C, M-R} and references therein.

There are interesting norm inequalities connected with characterizations of inner product spaces. One of celebrated characterizations of
inner product spaces has been based on the so-called Dunkl--Williams inequality.
In 1936, Clarkson \cite{Clark} introduced the concept of angular distance between nonzero elements $x$ and $y$ in a
normed space $(\mathscr{X},\|.\|)$ as $\alpha[x,y]=\left\|\frac{x}{\|x\|}-\frac{y}{\|y\|}\right\|$. One can observe some analogies between
 this notion and the concept of angle $A(x,y)$ between two nonzero vectors $x,y$ in a normed linear $(\mathscr{X},\|\cdot\|)$ defined by
$$ A(x,y)=\cos^{-1}\left[\tfrac12 \left(2-\left\|\frac x{\|x\|}-\frac y{\|y\|} \right\|^2\right)\right].$$ In \cite{F-D-A}, Freese, Diminnie
and Andalafte obtained a characterization of real inner product spaces in terms of their above notion of angle.
In 1964, Dunkl and Williams \cite{D-W} obtained a useful upper bound for the angular distance. They showed that
$$\alpha[x,y]\leq\frac{4\|x-y\|}{\|x\|+\|y\|}\ .$$

In the same paper, the authors proved that the constant $4$ can be replaced by $2$ if $\mathscr{X}$ is an inner product space.
Kirk and Smiley \cite{K-S} showed that
$$\alpha[x,y]\leq\frac{2\|x-y\|}{\|x\|+\|y\|}$$
characterizes inner product spaces.\\
In 1990, Al-Rashed \cite{A-R} generalized the work of Kirk and Smiley. He proved that in a real normed space $(\mathscr{X},\|.\|)$ the
following inequality
\begin{eqnarray*}
\alpha[x,y]\leq2^{\frac{1}{q}}\frac{\|x-y\|}{(\|x\|^{q}+\|y\|^{q})^{\frac{1}{q}}}\ \ \ (q\in(0,1])
\end{eqnarray*}
holds if and only if the given norm is induced by an inner product.\\
In \cite{Malig}, Maligranda considered the $p-$angular distance ($p\in \mathbb{R}$) as a generalization of the concept of
angular distance to which it reduces when $p=0$ as follows:
$$\alpha_{p}[x,y]:=\left\|\frac{x}{\|x\|^{1-p}}-\frac{y}{\|y\|^{1-p}}\right\|$$
Maligranda in the same paper and Dragomir in \cite{Drag} obtained some upper and lower bounds for the $p-$angular distance in normed spaces.\\

In this paper we present a new characterization of inner product spaces related to the $p-$angular distance.
We also generalize some results due to Dunkl, Williams, Kirk, Smiley and Al-Rashed by using the notion of $p-$angular
distance instead of that of angular distance.

%------------------------------------------------------------------------------%

\section{Main results}

We start this section with a norm inequality due to Maligranda \cite{Malig} that provides a suitable upper bound for
the $p-$angular distance.

\begin{theorem} \label{t1} \cite{Malig}
Let $(\mathscr{X},\|.\|)$ be a normed space and $p\in[0,1]$.Then
$$\alpha_{p}[x,y]\leq (2-p)\frac{\|x-y\|}{(\max\{\|x\|,\|y\|\})^{1-p}} \ \ (x,y\neq0)$$
\end{theorem}

The next theorem is a generalization of the Dunkl--Williams inequality \cite{D-W} and a theorem of Al-Rashed \cite[Theorem 2.2]{A-R}.
\begin{theorem} \label{t2}
Let $(\mathscr{X},\|.\|)$ be a real normed space, $p\in[0,1]$ and $q>0$.\\
Then the following inequality holds
\begin{eqnarray*}
\alpha_{p}[x,y]\leq 2^{1+\frac{1}{q}} \frac {\|x-y\|}{(\|x\|^{(1-p)q}+\|y\|^{(1-p)q})^\frac{1}{q}}
\end{eqnarray*}
for all nonzero elements $x$ and $y$ in $\mathscr{X}$.
\end{theorem}
\begin{proof}
Due to Theorem \ref{t1}, it is sufficient to show that
$$(2-p)\frac{\|x-y\|}{(\max\{\|x\|,\|y\|\})^{1-p}}\leq 2^{1+\frac{1}{q}} \frac {\|x-y\|}{(\|x\|^{(1-p)q}+\|y\|^{(1-p)q})^\frac{1}{q}}$$
Without loss of generality, we assume that $\|x\|\leq\|y\|$.\\ Since $p\leq1$ and $q>0$, we observe that $\|x\|^{(1-p)q}+\|y\|^{(1-p)q}
\leq2\|y\|^{(1-p)q}$.\\
Thus $(\|x\|^{(1-p)q}+\|y\|^{(1-p)q})^{\frac{1}{q}}\leq2^{\frac{1}{q}}\|y\|^{1-p}$ or equivalently
$$\frac{1}{\|y\|^{1-p}}\leq\frac{2^{\frac{1}{q}}}{(\|x\|^{(1-p)q}+\|y\|^{(1-p)q})^{\frac{1}{q}}}\,,$$
whence
$$(2-p)\frac{\|x-y\|}{(\max\{\|x\|,\|y\|\})^{1-p}}\leq\frac{2\|x-y\|}{\|y\|^{1-p}}\leq
2^{1+\frac{1}{q}}\frac{\|x-y\|}{(\|x\|^{(1-p)q}+\|y\|^{(1-p)q})^\frac{1}{q}}.$$
\end{proof}

\begin{proposition} \label{t3}
Let $(\mathscr{X},\|.\|)$ be an inner product space. Then the following inequality holds
$$\alpha_{p}[x,y]\leq 2\frac {\|x-y\|}{\|x\|^{1-p}+\|y\|^{1-p}} \ \ (x,y\neq0,\ \ p\in[0,1])\,.$$
\end{proposition}

\begin{proof}
Let $\langle\cdot,\cdot\rangle$ be the inner product on $\mathscr{X}$. Then

\begin{eqnarray}\label{L1}
\alpha_{p}^{2}[x,y]&=&\langle\frac{x}{\|x\|^{1-p}}-\frac{y}{\|y\|^{1-p}}\ ,\ \frac{x}{\|x\|^{1-p}}-\frac{y}{\|y\|^{1-p}}\rangle\nonumber\\
                   &=&\|x\|^{2p}-\frac{2{\rm Re}\langle x,y\rangle}{\|x\|^{1-p}\|y\|^{1-p}}+\|y\|^{2p}\nonumber\\
&=&\|x\|^{2p}-\frac{\|x\|^{2}+\|y\|^{2}-\|x-y\|^{2}}{\|x\|^{1-p}\|y\|^{1-p}}+\|y\|^{2p}\,.
\end{eqnarray}

Due to equality \eqref{L1} it is enough to show that
$$\|x\|^{2p}-\frac{\|x\|^{2}+\|y\|^{2}-\|x-y\|^{2}}{\|x\|^{1-p}\|y\|^{1-p}}+\|y\|^{2p}\leq4\frac{\|x-y\|^{2}}{(\|x\|^{1-p}+\|y\|^{1-p})^{2}}$$
or that the last inequality of the following sequence of equivalent inequalities holds.
$$\|x\|^{2p}-\frac{\|x\|^{2}+\|y\|^{2}}{\|x\|^{1-p}\|y\|^{1-p}}+\|y\|^{2p}
\leq(\frac{4}{(\|x\|^{1-p}+\|y\|^{1-p})^{2}}-\frac{1}{\|x\|^{1-p}\|y\|^{1-p}})\|x-y\|^{2}$$
$$\frac{\|x\|^{p+1}\|y\|^{1-p}-(\|x\|^{2}+\|y\|^{2})+\|x\|^{1-p}\|y\|^{p+1}}{\|x\|^{1-p}\|y\|^{1-p}}
\leq\frac{-(\|x\|^{1-p}-\|y\|^{1-p})^{2}\|x-y\|^{2}}{(\|x\|^{1-p}+\|y\|^{1-p})^{2}\|x\|^{1-p}\|y\|^{1-p}}$$
$$\frac{(\|x\|^{1-p}-\|y\|^{1-p})^{2}}{(\|x\|^{1-p}+\|y\|^{1-p})^{2}}\|x-y\|^{2}
\leq(\|x\|^{2}+\|y\|^{2})-(\|x\|^{p+1}\|y\|^{1-p}+\|x\|^{1-p}\|y\|^{p+1})$$
\begin{eqnarray}\label{L2}
\frac{(\|x\|^{1-p}-\|y\|^{1-p})^{2}}{(\|x\|^{1-p}+\|y\|^{1-p})^{2}}+\frac{\|x\|^{p+1}\|y\|^{1-p}+\|x\|^{1-p}\|y\|^{p+1}}{\|x-y\|^{2}}
\leq\frac{\|x\|^{2}+\|y\|^{2}}{\|x-y\|^{2}}
\end{eqnarray}

To prove \eqref{L2}, let $x,y\in \mathscr{X}-\{0\}$. Without loss of generality we suppose that $\|x\|<\|y\|$.
We define the differentiable real valued function $f$ as follows:
$$f(p)=\frac{(\|x\|^{1-p}-\|y\|^{1-p})^{2}}{(\|x\|^{1-p}+\|y\|^{1-p})^{2}}+
\frac{\|x\|^{p+1}\|y\|^{1-p}+\|x\|^{1-p}\|y\|^{p+1}}{\|x-y\|^{2}}\ \ (p\in[0,1]).$$
We claim that $f$ has exactly one local extremum point at the interval $(0,1)$.\\
By a straightforward calculation we see that
$$f^{\prime}(p)=0\Leftrightarrow4\|x-y\|^{2}(\|x\|^{1-p}-\|y\|^{1-p})+(\|y\|^{2p}-\|x\|^{2p})(\|x\|^{1-p}+\|y\|^{1-p})^{3}=0$$
$$\Leftrightarrow4b(1-a^{1-p})+(a^{2p}-1)(1+a^{1-p})^{3}=0 ,$$
where $a=\frac{\|y\|}{\|x\|}$ and $b=\frac{\|x-y\|^{2}}{\|x\|^{2}}$. Clearly $a>1$ and $(a-1)^{2}\leq b\leq(a+1)^{2}$.\\
Using the software MAPLE 11 we observe that the exponential equation
\begin{eqnarray*}
4b(1-a^{1-p})+(a^{2p}-1)(1+a^{1-p})^{3}=0
\end{eqnarray*}
has exactly one solution $p_{0}$ in the interval $(0,1)$. In fact the function $f$ takes the local minimum
at the point of $p_{0}$ due to the facts that $f^{\prime}(0)<0$ and $f^{\prime}(1)>0$.
Hence the function $f$ takes the absolute maximum at the boundary points of $[0,1]$.\\
Therefore
$$f(p)\leq\max\{f(0),f(1)\}\ \ \ (p\in[0,1]).$$
Thus $$f(p)\leq\max\left\{\frac{(\|x\|-\|y\|)^{2}}{(\|x\|+\|y\|)^{2}}+\frac{2\|x\|\|y\|}{\|x-y\|^{2}}\ ,
\frac{\|x\|^{2}+\|y\|^{2}}{\|x-y\|^{2}}\right\}\ \ \ (p\in[0,1]),$$
whence $$f(p)\leq\frac{\|x\|^{2}+\|y\|^{2}}{\|x-y\|^{2}}\ \ \ (p\in[0,1]).$$
\end{proof}

The next theorem is due to Lorch \cite{Lorch}, in which the dimension of the underlying space $\mathscr{X}$ plays no role. This is significant
since, for instance, the symmetry of Birkhoff--James orthogonality which is a characterization of inner product spaces is valid
when $\dim \mathscr{X}\geq3$, see \cite{Day, James1}. We recall that the behavior of a space in dimension $1$ or $2$ differs from that in dimension $3$, see \cite{Amir, RAT}.

\begin{theorem} \label{t4} \cite{Lorch}
 Let $(\mathscr{X},\|.\|)$ be a real normed space. Then the following statements are mutually equivalent:\\
{\rm (i)} For each $x,y\in\mathscr{X}$ if $\|x\|=\|y\|$, then $\|x+y\| \leq \|\gamma x+\gamma^{-1}y\|$ (for all $\gamma \neq0$).\\
{\rm (ii)} For each $x,y\in\mathscr{X}$ if $\|x+y\| \leq \|\gamma x+\gamma^{-1}y\|$ (for all $\gamma \neq0$), then $\|x\|=\|y\|$.\\
{\rm (iii)} $(\mathscr{X},\|.\|)$ is an inner product space.
\end{theorem}

The next result is an extension of the results of Al-Rashed \cite{A-R}. It provides a reverse of Proposition \ref{t3}.

\begin{theorem} \label{t5}
Let $(\mathscr{X},\|.\|)$ be a real normed space and $p\in[0,1)$.
If there exists a positive number $q$ such that
\begin{eqnarray}\label{L3}
\alpha_{p}[x,y]\leq 2^{\frac{1}{q}} \frac {\|x-y\|}{(\|x\|^{(1-p)q}+\|y\|^{(1-p)q})^\frac{1}{q}} \ \ (x,y\neq0)
\end{eqnarray}
Then $(\mathscr{X},\|.\|)$ is an inner product space.
\end{theorem}
\begin{proof}In the case when $p=0$ the theorem holds by a result due to Al-Rashed \cite[Theorem 2.4]{A-R}. So let us assume that
$0<p<1$.\\
Let $x,y\in\mathscr{X}$, $\|x\|=\|y\|$ and $\gamma \neq0$. From Theorem \ref{t4} it is enough to prove that
$\|x+y\| \leq \|\gamma x+\gamma^{-1}y\|$. Also we may assume that
$x\neq0$ and $y\neq0$.\\ Applying inequality \eqref{L3} to
$\gamma^{p^{n}}x$ and $-\gamma^{-p^{n}}y$ instead of $x$ and $y$,
respectively, we obtain
$$\alpha_{p}[\gamma^{p^{n}}x,-\gamma^{-p^{n}}y]\leq 2^{\frac{1}{q}} \frac {\|\gamma^{p^{n}}x+\gamma^{-p^{n}}y\|}{(\|\gamma^{p^{n}}x\|^{(1-p)q}+\|\gamma^{-p^{n}}y\|^{(1-p)q})^\frac{1}{q}}\,\,\,\,\,(n\in \mathbb{N}\cup\{0\})\,.$$
For $\gamma>0$ it follows from the definition of $\alpha_{p}$ that
$$\left\| \frac{\gamma^{p^{n}}x}{\gamma^{p^{n}(1-p)}\|x\|^{1-p}} + \frac{\gamma^{-p^{n}}y}{\gamma^{-p^{n}(1-p)}\|y\|^{1-p}}  \right\| \leq 2^{\frac{1}{q}} \frac {\|\gamma^{p^{n}}x+\gamma^{-p^{n}}y\|}{\|x\|^{1-p}(\gamma^{p^{n}(1-p)q}+\gamma^{-p^{n}(1-p)q})^\frac{1}{q}}$$
or equivalently
$$(\frac{\gamma^{p^{n}(1-p)q}+\gamma^{-p^{n}(1-p)q}}{2})^\frac{1}{q}\|\gamma^{p^{n+1}}x+\gamma^{-p^{n+1}}y\| \leq \|\gamma^{p^{n}}x+\gamma^{-p^{n}}y\|$$
for all $n\in \mathbb{N}\cup\{0\}$, whence $0\leq\|\gamma^{p^{n+1}}x+\gamma^{-p^{n+1}}y\| \leq
\|\gamma^{p^{n}}x+\gamma^{-p^{n}}y\|\,\,\,\,\,(n\in
\mathbb{N}\cup\{0\})$, since
$\gamma^{p^{n}(1-p)q}+\gamma^{-p^{n}(1-p)q}\geq2$. Hence $\{
\|\gamma^{p^{n}}x+\gamma^{-p^{n}}y\|\}_{n=0}^{\infty}$ is a
convergent sequence of nonnegative real numbers.
Thus we get $$\|x+y\|=\lim_{n\rightarrow\infty}\|\gamma^{p^{n}}x+\gamma^{-p^{n}}y\| \leq \|\gamma x+\gamma^{-1} y\|$$ due to $0<p<1$.\\
Now let $\gamma$ be negative. Put $\mu=-\gamma>0$. From the positive case we get $$\|x+y\|\leq\|\mu x+\mu^{-1}y\|=\|\gamma x+\gamma^{-1} y\|.$$
\end{proof}

\begin{lemma}\label{lemma}
Let $(\mathscr{X},\|.\|)$ be a normed space and $p\in[0,1]$.
If $0<q_{1}\leq q_{2}$, then $$2^{\frac{1}{q_{2}}} \frac {\|x-y\|}{(\|x\|^{(1-p)q_{2}}+\|y\|^{(1-p)q_{2}})^\frac{1}{q_{2}}}\leq2^{\frac{1}{q_{1}}}
 \frac {\|x-y\|}{(\|x\|^{(1-p)q_{1}}+\|y\|^{(1-p)q_{1}})^\frac{1}{q_{1}}}\ \ (x,y\neq0)$$
\end{lemma}
\begin{proof}
Without loss of generality, assume that $x\neq y$ . We have the following equivalent statements
$$2^{\frac{1}{q_{2}}}\frac {\|x-y\|}{(\|x\|^{(1-p)q_{2}}+\|y\|^{(1-p)q_{2}})^\frac{1}{q_{2}}}\leq2^{\frac{1}{q_{1}}}
\frac {\|x-y\|}{(\|x\|^{(1-p)q_{1}}+\|y\|^{(1-p)q_{1}})^\frac{1}{q_{1}}}$$
\begin{eqnarray*}
&\Leftrightarrow&(\|x\|^{(1-p)q_{1}}+\|y\|^{(1-p)q_{1}})^{\frac{1}{q_{1}}}\leq2^{\frac{1}{q_{1}}-\frac{1}{q_{2}}}(\|x\|^{(1-p)q_{2}}
+\|y\|^{(1-p)q_{2}})^\frac{1}{q_{2}}\\
&\Leftrightarrow&\|x\|^{(1-p)q_{1}}+\|y\|^{(1-p)q_{1}}\leq2^{1-\frac{q_{1}}{q_{2}}}(\|x\|^{(1-p)q_{2}}+\|y\|^{(1-p)q_{2}})^{\frac{q_{1}}{q_{2}}}\\
&\Leftrightarrow&(\|x\|^{(1-p)q_{2}})^{\frac{q_{1}}{q_{2}}}+(\|y\|^{(1-p)q_{2}})^{\frac{q_{1}}{q_{2}}}\leq2^{1-\frac{q_{1}}{q_{2}}}(\|x\|^{(1-p)q_{2}}
+\|y\|^{(1-p)q_{2}})^{\frac{q_{1}}{q_{2}}}\\
\end{eqnarray*}
The last inequality is an application of the following known inequality
$$a^{t}+b^{t}\leq2^{1-t}(a+b)^{t},\ \ (a,b\geq0,\ 0<t\leq1)$$
to $a=\|x\|^{(1-p)q_{2}},\ b=\|y\|^{(1-p)q_{2}}$ and $t=\frac{q_{1}}{q_{2}}$.
\end{proof}

Finally we are ready to state the characterization of inner product spaces. It is a generalization of a known theorem Kirk--Smiley \cite{K-S}.

\begin{theorem}\label{t6}
Let $(\mathscr{X},\|.\|)$ be a real normed space, and $p\in[0,1)$.
Then the following statements are mutually equivalent:\\
{\rm (i)} $\alpha_{p}[x,y]\leq 2^{\frac{1}{q}} \frac {\|x-y\|}{(\|x\|^{(1-p)q}+\|y\|^{(1-p)q})^\frac{1}{q}}\ \ (x,y\neq0)$, for all $q\in(0,1]$.\\
{\rm (ii)} $\alpha_{p}[x,y]\leq 2^{\frac{1}{q}} \frac {\|x-y\|}{(\|x\|^{(1-p)q}+\|y\|^{(1-p)q})^\frac{1}{q}}\ \ (x,y\neq0)$, for some $q>0$.\\
{\rm (iii)} $(\mathscr{X},\|.\|)$ is an inner product space.
\end{theorem}
\begin{proof}
${\rm (i)}\Rightarrow {\rm (ii)}$ is trivial.\\
${\rm (ii)} \Rightarrow {\rm (iii)}$ is the same as Theorem \ref{t5}.\\
To complete the proof, we need to establish the implication ${\rm (iii)}\Rightarrow {\rm (i)}$. To see this, let $q\in(0,1]$ be arbitrary. It follows from Proposition \ref{t3} that
\begin{eqnarray}\label{L6}
\alpha_{p}[x,y]\leq 2\frac {\|x-y\|}{\|x\|^{1-p}+\|y\|^{1-p}}\ \ (x,y\neq0)
\end{eqnarray}
By setting $q_{1}=q$ and $q_{2}=1$ in the Lemma \ref{lemma} we get
\begin{eqnarray}\label{L7}
2\frac {\|x-y\|}{\|x\|^{1-p}+\|y\|^{1-p}}\leq2^{\frac{1}{q}} \frac {\|x-y\|}{(\|x\|^{(1-p)q}+\|y\|^{(1-p)q})^\frac{1}{q}}
\end{eqnarray}
Now the result follows from inequalities \eqref{L6} and \eqref{L7}.
\end{proof}
%----------------------------------------------------------------------------------------------------------------------------%

\bibliographystyle{amsplain}

\begin{thebibliography}{99}


\bibitem{A-R} A.M. Al-Rashed, \textit{Norm inequalities and characterizations of inner product spaces},
J. Math. Anal. Appl. \textbf{176} (1993), 587--593.

\bibitem{A-T} C. Alsina and M.S. Tom\'as, \textit{On parallelogram areas in normed linear spaces},  Aequationes Math. \textbf{72} (2006), no. 3,
 234--242.

\bibitem{Amir} D. Amir, \textit{Characterizations of inner product spaces}, Operator Theory: Advances and Applications,
20. Birkh\"{a}user Verlag, Basel, 1986.

\bibitem{B-C} M. Baronti and E. Casini, \textit{Characterizations of inner product spaces by orthogonal vectors}, J. Funct. Spaces Appl. \textbf{4}
 (2006), no. 1, 1--6.


\bibitem{Clark} J.A. Clarkson, \textit{Uniformly convex spaces}, Trans. Amer. Math. Soc. \textbf{40} (1936), 396--414.

\bibitem{Day} M.M. Day, \textit{Some characterizations of inner-product spaces},
Trans. Amer. Math. Soc. \textbf{62} (1947), 320--337.

\bibitem{Drag} S.S. Dragomir, \textit{Inequalities for the $p$-angular distance in normed linear spaces},
Math. Inequal. Appl. \textbf{12} (2009), no. 2, 391--401.

\bibitem{D-W} C.F. Dunkl and K.S. Williams, \textit{Mathematical Notes: A Simple Norm Inequality},
Amer. Math. Monthly \textbf{71} (1964), no. 1, 53--54.

\bibitem{FRE} M. Fr\'echet, \textit{Sur la definition axiomatique d'une classe
d'espaces vectoriels distanci\'es applicables vectoriellement sur
l'espace de Hilbert}, Ann. of Math. \textbf{(2) 36} (1935), no. 3,
705--718.

\bibitem{F-D-A} R.W. Freese, C.R. Diminnie and E.Z. Andalafte, \textit{Angle bisectors in normed linear spaces},
Math. Nachr. \textbf{131} (1987), 167--173.


\bibitem{James1} R.C. James, \textit{Inner products in normed linear spaces},
Bull. Amer. Math. Soc. (N.S.) \textbf{53} (1947), 559--566.


\bibitem{J-V} P. Jordan and J. von Neumann, \textit{On inner products in linear, metric spaces}, Ann. of Math. \textbf{(2) 36} (1935), no. 3, 719--723.

\bibitem{K-S} W.A. Kirk and M.F. Smiley, \textit{Mathematical Notes: Another characterization of inner product spaces},
Amer. Math. Monthly \textbf{71} (1964), no. 8, 890--891.


\bibitem{Lorch} E.R. Lorch, \textit{On certain implications which characterize Hilbert space}, Ann. of Math. (3) \textbf{49} (1948), 523--532.


\bibitem{Malig} L. Maligranda, \textit{Simple norm inequalities}, Amer. Math. Monthly \textbf{113} (2006), no. 3, 256--260.


\bibitem{M-R} M.S. Moslehian and J.M. Rassias, \textit{A characterization of inner product spaces concerning an Euler-Lagrange identity}, Commun.
Math. Anal. \textbf{8} (2010), no. 2, 16--21.

\bibitem{RAT} J. R\"atz, \textit{Charakterisierung von Skalarprodukträumen: Dimension 2 versus Dimension $\geq 3$. (German) [Characterization of scalar
 product spaces: dimension 2 versus dimension $\geq 3$]}, Festkolloquium für Ludwig Reich (Graz, 2000),  41--46, Grazer Math. Ber., 342,
 Karl-Franzens-Univ. Graz, Graz, 2000.





\end{thebibliography}

\end{document}